\documentclass[11pt]{article}
  \usepackage{a4wide}
  \usepackage{amsmath}
  \usepackage{amssymb}
  \usepackage{latexsym}
  \usepackage[pdftex]{graphicx}
  \usepackage{subcaption}
  \usepackage{color}
  \usepackage[mathscr]{euscript}

  \newenvironment{proof}{\vspace{1ex}\noindent{\bf Proof:}}{\hspace*{\fill}$\blacksquare$\vspace{1ex}}
  \newenvironment{proofof}[1]{\vspace{1ex}\noindent{\bf Proof of #1:}}{\hspace*{\fill}$\blacksquare$\vspace{1ex}}

  \newtheorem{theorem}{Theorem} 
  \newtheorem{lemma} [theorem] {Lemma}
  \newtheorem{corollary} [theorem] {Corollary}
 
\newcommand{\eN}[0]{\ensuremath{ \mathbb N}}
\newcommand{\Zed}[0]{\ensuremath{ \mathbb Z}}

\newcommand{\Pee}[0]{\ensuremath{{\mathbb P}}}

\newcommand{\Ee}[0]{\ensuremath{{\mathbb E}}}

\newcommand{\isd}[0]{\hspace{.2ex} \raisebox{-.1ex}{$=$} \hspace{-1.5ex} 
\raisebox{1ex}{{$\scriptstyle d$}} \hspace{.8ex} }

 \newcommand{\eps}{\varepsilon}

\DeclareMathOperator{\mood}{mod}

\DeclareMathOperator{\Bi}{Bi}

\DeclareMathOperator{\FO}{FO}

\DeclareMathOperator{\spec}{spec}

\DeclareMathOperator{\InC0}{InC_0}
\DeclareMathOperator{\Hedge}{Edge}
\DeclareMathOperator{\Bigger}{Bgr}

\DeclareMathOperator{\CN}{CmNb}
\DeclareMathOperator{\UniP}{UniP}
\definecolor{orange}{RGB}{255,127,0}
\definecolor{pink}{RGB}{255,150,150}

\begin{document}

\title{The first order convergence law fails for random perfect graphs}


\author{Tobias M\"uller\thanks{Bernoulli Institute, 
Groningen University, {\tt tobias.muller@rug.nl}. 
Supported in part by the Netherlands Organisation for Scientific Research (NWO) under project nos 612.001.409 and 639.032.529.
Part of this research was carried out while this author visited Barcelona supported by the BGSMath research program on Random Discrete Structures and Beyond 
that was held in Barcelona from May to June 2017.} 
\and Marc Noy\thanks{Department of Mathematics, Universitat Politècnica de Catalunya, Barcelona Graduate School of Mathematics, {\tt marc.noy@upc.edu}. 
Supported in part by grants MTM2014-54745-P and MDM-2014-0445}}
\date{\today}

\maketitle

\begin{abstract}
We consider first order expressible properties of random perfect graphs. That is, we pick a graph $G_n$ uniformly at random from all 
(labelled) perfect graphs on $n$ vertices and consider the probability that it satisfies some graph property that can be expressed
in the first order language of graphs.
We show that there exists such a first order expressible property for which the probability that $G_n$ satisfies it does not 
converge as $n\to\infty$.
\end{abstract}

%
%
%

\section{Introduction}

A graph is perfect if the chromatic number equals the clique number in each of its induced subgraphs.
Perfect graphs are a central topic in graph theory and play an important role in combinatorial optimization. 
In this paper we will study the random graph chosen uniformly at random from all (labelled) perfect graphs on $n$ 
vertices.
The first thing one might want in order to prove results about this object is a mechanism for generating random perfect 
graphs that is more descriptive than ``put all $n$-vertex perfect graphs in a bag and pick one uniformly at random''.
Such a mechanism has been introduced recently by McDiarmid and Yolov \cite{colin}. 
Before presenting it, let us discuss as a preparation a simpler subclass of perfect graphs.

A graph is chordal if every cycle of length four or more has a chord, that is, an edge joining  non-consecutive vertices in the cycle. 
A graph is split if its vertex set can be partitioned into a clique and an independent set (with arbitrary edges across the partition). 
It is easy to see that a split graph is chordal, but not conversely. On the other hand, it is known that almost all chordal graphs are 
split~\cite{chordal}, in the sense that the proportion of chordal graphs that are split tends to 1 as the number of vertices $n$ tends to infinity. 
Thus we arrive at a very simple process for generating random chordal graphs: (randomly) partition the vertex set into  a clique $A$ and an independent 
set $B$, and add an arbitrary set of edges between $A$ and $B$ (chosen uniformly at random from all posible sets of edges between $A$ and $B$). 
The distribution we obtain in this way is not uniform as a split graph may arise from different partitions into a clique and an independent set, but 
it can be seen that when the size of $A$ is suitably sampled then its total variational distance to the uniform distribution tends to zero 
as $n$ tends to infinity. 
Now we turn to random perfect graphs. 
A graph $G$ is \emph{unipolar} if for some $k\ge 0$ its vertex set $V(G)$ can be partitioned
into $k + 1$ cliques $C_0, C_1, \dots,C_k$, so that there are no edges between $C_i$ and $C_j$
for $ 1 \le i < j \le k$. Following \cite{colin} we call $C_0$ the \emph{central clique}, and the $C_i$ for $i\ge 1$  the \emph{side cliques}. 
A graph $G$ is \emph{co-unipolar} if the  complement $\overline{G}$  is unipolar; and it is a \emph{generalized split} graph if it is unipolar or co-unipolar. Notice that a graph can be both unipolar and co-unipolar, and that when the $C_i$ for $i\ge 1$ are reduced to a single vertex, a generalized split  graph is split. It can be shown that generalized split graphs are perfect, and it was proved in \cite{PSperfect} that almost all perfect graphs are generalized split.

McDiarmid and Yolov  \cite{colin} have devised the following process for generating random unipolar graphs. 
Choose an integer $m \in [n]$ according to a suitable distribution; choose a random $m$-subset $C_0 \subseteq [n]$; choose a (unifromly) random set partition  $[n] \backslash C_0=C_1 \cup \cdots \cup C_k$ of the complement and make all the $C_i$ into cliques; finally add edges between $C_0$ and $[n] \backslash C_0$ independently with probability $1/2$, and no further edges.
Again this scheme is not uniform but it is shown in \cite{colin} that it approximates the uniform distribution on unipolar graphs on $n$ vertices
up to total variational distance $o_n(1)$. 
This gives a useful  scheme for random perfect graphs: pick a random unipolar graph $G$ on $n$ vertices according to the previous scheme, and 
flip a fair coin: if the coin turns up heads then output $G$, otherwise output its complement $\overline{G}$. 
Several properties of random perfect graphs are proved in \cite{colin} using this scheme. One notable such result is that for every 
fixed graph $H$ the probability that the random perfect graph on $n$ vertices has an induced subgraph isomorphic to $H$
tends to a limit that is either $0,1/2$ or $1$.

In this paper we consider graph properties that can be expressed in the {\em first order language of graphs} ($\FO$), on random perfect graphs.
Formulas in this language are constructed using variables $x, y, \dots$ ranging over
the vertices of a graph, the usual quantifiers $\forall, \exists$, the usual logical connectives $\neg, \vee, \wedge$, etc., parentheses and the binary relations $=, \sim$,
where $x\sim y $ denotes that $x$ and $y$ are adjacent.
To aid readability we will also use commas and semicolons in the formulas in this paper.
In $\FO$ one can for instance write ``$G$ is triangle-free" as 
$\neg\exists x,y,z: (x\sim y)\wedge(x\sim z)\wedge(y\sim z)$.
We say that a graph $G$ is a model for the sentence $\varphi \in \FO$ if $G$ satisfies $\varphi$, and write $G\models \varphi$.
(A sentence is a formula in which every variable is ``bound'' to a quantifier.)

Several restricted classes of graphs  have been studied with respect to the limiting behaviour of $\FO$ properties,
and usually  one proves either a zero-one law (that is, every $\FO$ property has limiting probability $\in\{0,1\}$) or a convergence law
(that is, every $\FO$ property has a limiting probability). 
For instance, a zero-one law has been proved for trees \cite{McColm} and for graphs not containing a clique of fixed size \cite{Kt-free}, 
while a convergence law has been proved for $d$-regular graphs for fixed $d$ \cite{lynch}, and for forests and planar graphs \cite{minor-closed}.

In the light of the above mentioned result of McDiarmid and Yolov on the limiting probability of containing a fixed induced subgraph, one might 
expect the convergence law to hold for random perfect graphs, perhaps even with the limiting probabilities only taking the values $0,1/2,1$.
The main result of this paper however states something rather different is the case.

\begin{theorem}\label{thm:FO}
	There exists a sentence $\varphi \in \FO$ such that
	$$ \lim_{n\to\infty} \Pee\left[ P_n \models \varphi \right] \text{ does not exist, } $$
	where $P_n$ is chosen uniformly at random from all (labelled) perfect graphs on $n$ vertices.
\end{theorem}

This is in stark contrast to random chordal graphs. The scheme we discussed above based on random  split graphs is in fact  
very similar to the binomial bipartite random graph with independent edge probabilities  equal to $1/2$.  
A standard argument shows that in fact a zero-one law holds in this case, that is, the limiting probability that a $\FO$ property 
is satisfied tends either  to 0 or 1 as $n \to \infty$ \cite{spencer}.
  
Our proof of Theorem \ref{thm:FO} draws on the techniques introduced in the proof of the celebrated Shelah-Spencer result of non-convergence 
in the classical $G(n,p)$ model when $p=n^{-\alpha}$ and $\alpha \in (0,1)$ is a rational number~\cite{ShelahSpencer} (see also \cite{spencer}). 
In fact, it is the richness of unipolar graphs together with the properties of random set partitions that allow us to produce a non-convergent 
first order sentence.  \\ 


In addition we prove the following undecidability result. 

\begin{theorem}\label{thm:undecidable}
	There does not exist an algorithm that, given as input a $\varphi \in \FO$ that is guaranteed to
	have either limiting probability zero or limiting probability one, decides whether the limiting probability 
	equals one.
	
\end{theorem}

For more discussion and open problems we refer the reader to Section~\ref{sec:discuss}

\section{Preliminaries}

Throughout this paper, we will say that a sequence of events $E_1, E_2, \dots$ holds {\em with high probability} 
if $\lim_{n\to\infty} \Pee( E_n ) = 1$.

Recall that the log-star function 

\[ \log^* n := \min\{ k \in \Zed_{\geq 0} : T(k) \geq n \}, \] 

\noindent
is the least integer $k$ for which $T(k)$ is at least $n$, where $T(.)$ denotes the tower function -- which can be 
defined recursively by $T(0)=1$ and $T(n+1) = 2^{T(n)}$. Put differently, $T(n)$ is a ``tower'' of $2$s of height $n$ and 
$\log^* n$ is the number of iterations of the base two logarithm that are needed to reduce $n$ to one or less.

The {\em spectrum}  $\spec(\varphi)$ of a sentence $\varphi \in \FO$ is the set of all $n \in \eN$ for which there exists 
a graph on $n$ vertices that satisfies $\varphi$, that is
$$ \spec(\varphi) := \{ v(G) : G \models \varphi \}. $$

The following lemma is a straightforward adaptation of a construction of Shelah and Spencer~\cite{ShelahSpencer}

\begin{lemma}\label{lem:logloglemma}
There exist $\varphi_0, \varphi_1 \in \FO$ such that 
$$ 
\renewcommand*{\arraystretch}{1.7}
\begin{array}{ll}
&\log^* n \mood 100 \in \{2, \dots 49\} \Rightarrow n \in \spec(\varphi_0) \setminus \spec(\varphi_1),\\  
&\log^* n \mood 100 \in \{52, \dots 99\} \Rightarrow n \in \spec(\varphi_1) \setminus \spec(\varphi_0).  \\
\end{array} 
$$
\end{lemma}
We remark that $\varphi_0$ is constructed explicitly in~\cite{spencer}, pages 112--113, and that  $\varphi_1$ is a straightforward adaptation 
of this construction.

We also need the following  consequence of a more general theorem of Trakhtenbrot \cite{trakhtenbrot} (see also \cite[page 303]{purple-book}) on undecidability 
in first order logic. 

\begin{theorem}[Trakhtenbrot]
There does not exist an algorithm that can decide, given an arbitrary sentence $\varphi \in \FO$ as input, whether 
$\spec(\varphi) = \emptyset$. \\
(That is, whether or not $\varphi$ is satisfied by at least one finite graph.)
\end{theorem}

We next recall the scheme from \cite{colin} for generating random unipolar graphs with $n$ vertices, together with some key properties of the construction. 
\begin{itemize}
\item 
Choose the size $m$ of the central clique $C_0$ according to a  distribution proportional to $\binom{n}{m}2^{m(n-m)}B(n-m)$, where 
$B(n)$ is the $n$-th Bell number (the exact distribution is not needed  and is shown only for completeness), and choose $C_0$ as a random 
$m$-subset of $[n]$. 
\item Take a (uniformly) random set partition of the complement 
$[n] \backslash C_ 0 = C_1 \cup \cdots \cup C_k,$ and make each $C_i$  a clique.
\item Add edges arbitrarily between $C_0$ and $[n] \backslash C_0$ (i.e.~add a set of edges chosen uniformly at random
from all possible sets of edges between $C_0$ and $[n]\setminus C_0$), and no further edges.
\end{itemize}

\noindent
A random perfect graph is obtained by taking a random unipolar graph $G_n$ as generated  above  and flipping a fair coin: if the coin turns up 
heads then take $G_n$, otherwise take its complement. 
It is proved in \cite{colin} that the probability that a uniformly random perfect graph is both unipolar and co-unipolar is exponentially small, and 
that the distribution obtained by the above scheme has total variation distance $o(1)$ to 
the uniform random perfect graph, hence it can be used to prove properties of the uniform random perfect graph.  

It is shown in \cite{colin} that the cliques $C_i$ satisfy the following properties with probability tending to one as $n$ tends to infinity:



\begin{enumerate}
  \item $|C_0| = \frac{n}{2}(1+o(1))$. This follows from \cite[Theorem 2.5]{colin}.
  \item Let    $r$ be the unique root of $r e^r = n - |C_0|$. 
  For $t=1,\dots, (e-\eps) \ln n$, with $\eps>0$ arbitrary but fixed, we have
  \begin{eqnarray}\label{eq:number-of-parts} 
   \left|\{ j : |C_j| = t \}\right| = \Omega\left( r^t / t! \right). 
   \end{eqnarray}
   This follows from the results in \cite[Section 2.2.3]{colin}.
\end{enumerate}

\noindent
We note that, with high probability, we have 
\begin{equation}\label{eq:rass} 
r = \ln n - (1+o(1)) \ln\ln n. 
\end{equation}

\section{Proofs}

We start by noticing that it is easy to tell in $\FO$ whether the random perfect graph is unipolar.

\begin{corollary}\label{cor:unip}
There is a sentence $\UniP \in \FO$ such that if $P_n$ denotes the random perfect graph then, with high probability,
$P_n \models \UniP$ if and only if $P_n$ is unipolar.
\end{corollary}

\begin{proof}
Let $H$ be any graph that is unipolar but not co-unipolar, and let 
$\UniP \in \FO$ formalize that ``$H$ is an induced subgraph''.
The conclusion follows from \cite[Theorem 2.3 and Lemma 4.1]{colin}, implying that $\Pee(P_n \models \UniP | P_n \text{ is unipolar } )= 1-o(1)$,
and $\Pee(P_n \models \UniP | P_n \text{ is not unipolar } )=o(1)$
(the second statement is clear since the class of co-unipolar graphs is closed under taking induced subgraphs).
\end{proof}

In what follows $G_n$ will denote the alternative scheme of random unipolar graphs of McDiarmid and Yolov.
In the light of the above, it is enough for us to show the statement of Theorem~\ref{thm:FO} for $G_n$ rather than $P_n$.
This is because if $\varphi$ is such that $\Pee( G_n \models \varphi )$ does not converge then 
$$ \Pee( P_n \models \UniP \wedge \varphi) = (1/2+o(1)) \cdot \Pee( G_n \models \varphi ) 
$$ 
also does not converge.
Similarly, it suffices to prove Theorem~\ref{thm:undecidable} for $G_n$ rather than $P_n$.
To see this, note that 
$$ \Pee( P_n \models (\UniP \wedge \varphi)\vee(\neg\UniP \wedge \overline{\varphi} ) ) = \Pee( G_n \models \varphi ) + o(1), 
$$
where $\overline{\varphi}$ is obtained from $\varphi$ by swapping $a\sim b$ for $\neg(a\sim b)$
(and hence also $\neg(a\sim b)$ is replaced with $\neg\neg(a\sim b)$ which is equivalent to $a\sim b$), so that 
$G \models \overline{\varphi}$ if and only if $\overline{G} \models \varphi$.

In the remainder of this section we will therefore only work with $G_n$.
In the proofs below, we usually think of revealing (conditioning on) the partition $C_0, \dots, C_k$ of $[n]$ so that all computations of 
probabilities etc.~will only be with respect to the random edges between $C_0$ and $\bigcup_{i>0} C_i$.
This is justified because in the construction, we add the edges between $C_0$ and its complement last, after 
the partition $C_0, \dots C_k$ has been chosen.

The following observation provides us a useful way to distinguish whether a vertex is in $C_0$ or not.

\begin{lemma}
With high probability it holds that, for each vertex $v$,

\bigskip
\centerline{$ v \in C_0$ if and only if~N(v) contains a stable set of size three.}
\end{lemma}

\begin{proof}
We first note that if $v \not\in C_0$ then $v\in C_i$ for some $i > 0$ and then its neighbourhood $N(v) \subseteq C_0 \cup C_i$ 
is covered by two cliques. So in particular $N(v)$ does not contain a stable set of size three.

For the reverse, we first observe that by construction and  estimates~\eqref{eq:number-of-parts} and~\eqref{eq:rass}, there is a constant $c>0$ such that 
with high probabibility there are at least $c \ln^2 n$ 
parts $C_j$ of size $|C_j|=2$.
Moreover, if a vertex $v \in C_0$ does not have a stable set of size three in its neighbourhood, then 
$v$ is adjacent to no more than four of the vertices in $\bigcup_{|C_j|=2} C_j$.
Thus, if we let $E$ denote the event that there is a $v \in C_0$ whose neighbourhood $N(v)$ does not contain a stable set of size three
then
$$ \begin{array}{rcl} 
\Pee( E ) & \leq & \Pee( |\{j:|C_j|=2\}| <  c \ln^2 n ) + n \cdot \Pee( \Bi( 2c\ln^2 n, 1/2 ) \leq 4 ) \\
& = & o(1) + n e^{-\Omega(\ln^2 n) } =  o(1),
\end{array} $$
where we have used the Chernoff bound. 
\end{proof}

\begin{corollary}
There exists an $\FO$-formula $\InC0$ with one free variable such that, with high probability,
$\InC0(x)$ holds for all $x\in C_0$ and $\neg\InC0(x)$ holds for all $x\not\in C_0$.
\end{corollary}

\begin{proof}
It is easily checked that the following formula states that the neighbourhood of $x$ contains a stable set of size three:
$$ 
\begin{array}{rcl}
 \InC0(x) & := & \exists x_1, x_2, x_3 : (x\sim x_1) \wedge (x\sim x_2) \wedge (x\sim x_3) \\
 & & \wedge \neg(x_1=x_2) \wedge \neg(x_1=x_3) \wedge \neg(x_2=x_3) \\ 
 & & \wedge \neg(x_1\sim x_2) \wedge \neg(x_1\sim x_3) \wedge \neg(x_2\sim x_3).
\end{array} 
$$
\end{proof}

\noindent
For $S \subseteq [n]$, we write $N(S) := \bigcap_{v\in S} N(v)$ for the set of common neighbours of $S$ 
in our random graph $G_n$.

\begin{corollary}
There exists an $\FO$-formula $\CN$ with two free variables such that, 
with high probability, $\CN(x,y)$ holds if and only if~$x\in C_0, y \in C_i$ for some $i>0$, and $x \in N(C_i)$. 
\end{corollary}

\begin{proof}
It is easily checked that the following definition works out (assuming that $\InC0$ expresses membership of $C_0$ as intended):
$$ \begin{array}{rcl} \CN(x,y) & := & \InC0(x) \wedge \neg \InC0(y) \wedge (x\sim y) \\ 
& &  \wedge (\forall z : (\neg \InC0(z) \wedge (z\sim y))\Rightarrow (x\sim z)).    
   \end{array} $$
\end{proof}

For $S, T \subseteq [n]$ we let $H(S,T)$ denote the graph with vertex set 
$S$ and an edge between $a, b \in S$ if and only if~there is a $v \in T$ that is adjacent to 
both $a$ and $b$.

\begin{corollary}
The exists an $\FO$-formula $\Hedge$ with three free variables such that, with high probability,
$\Hedge(x,y,z)$  holds if and only if $x,y \in C_0$, $x \neq y$,   $z \in C_i$ for some $i>0$, and $xy$ is an edge of $H(C_0,C_i)$.
\end{corollary}

\begin{proof}
It is easily checked that the following formula will do the trick (again assuming $\InC0$ expresses the right thing):
$$ \begin{array}{rcl} 
\Hedge(x,y,z) & := & 
\InC0(x) \wedge \InC0(y) \wedge \neg \InC0(z) \wedge \neg (x=y) \\
& & \wedge 
(\exists z_1 : ((z_1=z)\vee (\neg \InC0(z_1)\wedge (z_1\sim z)) \wedge
 (x\sim z_1)\wedge (y\sim z_1)).
\end{array} $$

(To aid the reader, let us point out that the only purpose of the variable $z$ is to represent $C_i$.)
\end{proof}

\begin{corollary}\label{cor:transform}
For every $\varphi \in \FO$ there exists an $\FO$-formula $\Phi(x,y)$ with two free variables such that, with high probability,
$\Phi(x,y)$ holds if and only if~$x \in C_i, y\in C_{j}$ for some $i,j>0$, and $H(N(C_i),C_j) \models \varphi$.
\end{corollary}

\begin{proof}
The formula $\Phi$ can be read off from $\varphi$ in a straightforward way as follows. 
In $\varphi$, we replace every occurrence of $a \sim b$ by $\Hedge(a,b,y)$ and 
we ``relativize the quantifiers to $\CN(.,x)$''. That is:
\begin{itemize} 
\item $\exists z : \psi$ is replaced by $\exists z : \CN(z,x) \wedge \psi$, and; 
\item $\forall z : \psi$ is replaced by $\forall z : \CN(z,x) \Rightarrow \psi$.
\end{itemize}
Finally we take the conjunction of the end result with $\neg\InC0(x) \wedge \neg\InC0(y)$.
The reader can easily verify that the formula we obtain is as required (assuming $\InC0$ and $\CN$ take on their
intended meanings).
\end{proof}

Let us write 
$$ \ell := \left\lceil\ln\ln\ln n\right\rceil. 
$$
\begin{lemma}\label{lem:rightsize}
With high probability, for every $0 \leq \ell' \leq \ell$, there exist $n^{\Omega(1)}$ indices $i>0$ with $|N(C_i)| = \ell'$.
\end{lemma}

\begin{proof}
Let $t \in \eN$ be such that $(n/2) \cdot (1/2)^{t-1} > \ell' \geq (n/2) \cdot (1/2)^t$.
So we have $(n/2) \cdot (1/2)^t \in (\ell'/2,\ell']$ and 
\begin{equation}\label{eq:tass} 
t = (1+o(1)) \log_2 n = (1+o(1)) \ln n / \ln 2. 
\end{equation}
%
Let us write $J := \{ j : |C_j| = t \}$.
In the McDiarmid-Yolov construction, with high probability, we have 
$$ 
\renewcommand*{\arraystretch}{1.2}\begin{array}{rcl} 
|J| 
& = & \Omega( r^t / t! ) \\
& = & \Omega\left( \exp\left[  t \ln r - t\ln t + t + O(\ln t) \right] \right)\\
& = & \Omega\left( \exp\left[  t \ln(r/t) + t + O(\ln t) \right] \right)\\
& = & \Omega\left( \exp\left[  t \cdot \left( \ln\ln 2 + 1 + o(1) \right) + O(\ln t) \right] \right)\\
& = & \exp\left[ \Omega( \ln n ) \right] \\
& = & n^{\Omega(1)},
\end{array} 
$$
where we have used Stirling's approximation in the second line, and we have used that $r/t = \ln 2 + o(1)$ by~\eqref{eq:rass} and~\eqref{eq:tass}
in the fourth line, and that $\ln\ln 2 + 1 \approx 0.633 > 0$ in the fifth line.

We have that (conditional on the partition $C_0, \dots, C_k$):
$$ |N(C_j)| \isd \Bi( |C_0|, (1/2)^t ) \quad (\forall j \in J). $$
Hence 
$$ \Ee |N(C_j)| = |C_0| (1/2)^t = (1+o(1)) (n/2) (1/2)^t = \Theta(\ell').
$$

\noindent
(The expectation again being conditional on the partition $C_0, \dots, C_k$).
Therefore, for each $j \in J$ (conditional on the partition $C_0, \dots, C_k$) we have
$$ 
\renewcommand*{\arraystretch}{1.5}
\begin{array}{rcl} 
   \Pee( |N(C_j)| = \ell' ) 
   & = & \displaystyle 
   {|C_0|\choose \ell'} (1/2)^{t\ell'} (1-(1/2)^t)^{|C_0|-\ell'} \\
   & \geq & \displaystyle 
   \left(|C_0|/\ell'\right)^{\ell'} (1/2)^{t\ell'} \left(1-(1/2)^t\right)^{|C_0|-\ell'} \\
   & = & \displaystyle 
   \left(|C_0| (1/2)^t/\ell'\right)^{\ell'} (1-(1/2)^t)^{|C_0|-\ell'} \\
   & = & \displaystyle 
   \left(\Theta(\ell')/\ell'\right)^{\ell'} (1-(1/2)^t)^{|C_0|-\ell'} \\
   & = & \displaystyle
   \Theta(1)^{\ell'}  (1-(1/2)^t)^{|C_0|-\ell'} \\
   & = & \displaystyle 
   \exp[ \pm O(\ell') + (|C_0|-\ell') \ln(1-(1/2)^t) ] \\
   & \geq & \displaystyle 
   \exp[ - O( \ell' ) - O( |C_0|(1/2)^t ) ] \\
   & \geq & \displaystyle 
   \exp[ - O( \ell' ) ]. 
   \end{array} 
   $$
Here we have used the standard bound ${n\choose k} \geq \left(n/k\right)^k$ in the second line, and the estimate 
$\ln(1-x) = -\Theta(x)$ as $x\downarrow 0$ in the seventh line.
Let us denote by $I := \{ j \in J : |N(C_j)| = \ell' \}$ the number of $j \in J$ for which $C_j$ has exactly $\ell'$
common neighbours. The previous considerations show that 
$$ \Ee |I| = |J| \cdot \exp[ - O(\ell' ) ] = \exp[ \Omega( \ln n ) - O( \ell' ) ] = \exp[ \Omega( \ln n ) ]
= n^{\Omega(1)}, $$
since $\ell' \leq \ell \ll \ln n$.
Let us now point out that the random variables $\{|N(C_j)| : j \in J\}$ are in fact independent (since they depend on disjoint
sets of edges -- of course this is again all conditional on the partition $C_0, \dots, C_k$).
So in particular $|I|$ is a binomial random variable, whose mean tends to infinity. Hence (for instance, by Chebyschev's inequality) 
$\Pee( |I| < \Ee |I| / 2) = o(1)$.
\end{proof}

\begin{lemma}\label{lem:realise}
With high probability, the following holds.
For every $i,j > 0$ such that $|N(C_i) \cup N(C_{j})| \leq  2\ell$, and 
for every (labelled) graph $G$ with $V(G) = N(C_i) \cup N(C_j)$, there is a $k>0$ such that
$H(N(C_i) \cup N(C_j),C_k) = G$.
\end{lemma}

\paragraph{Remark.} We emphasize that when we say $G=H$ we do not just speak about isomorphism, but we really mean that 
$V(G)=V(H)$ and $E(G)=E(H)$.

\begin{proof}
We set $t := \left\lceil \sqrt{\ln n}\right\rceil$, and let $K := \{ k : |C_k|= t \}$.
As before, with high probability, we have 
$$ 
\renewcommand*{\arraystretch}{1.2}
\begin{array}{rcl} 
|K| & = & \Omega( r^t / t! ) \\
& = & \Omega\left( \exp\left[ t \ln r - t\ln t + t + O(\ln t) \right] \right) \\
& = & \Omega\left( \exp\left[ t \ln(r/t) + t + O(\ln t) \right] \right) \\
& = & \exp\left[ \Omega( t \ln\ln n ) \right],
\end{array} $$

\noindent
where we have again used Stirling for the second line, and that $r/t = (1+o(1)) \sqrt{\ln n}$ by~\eqref{eq:rass} for the last line.

For the moment, let us fix some set $S \subseteq C_0$ of cardinality $\leq 2\ell$, a $k \in K$ and a ``target'' graph $G$ with $V(G) = S$.
Let $E_{S,G, k}$ denote the event that $H(S,C_k) = G$.
Since $|C_k|= t > {|S|\choose 2}$ there is at least one way to choose the edges between $S$ and $C_k$ that would result in 
desired situation where $H(S, C_k) = G$. In other words,
$$ \Pee( E_{S,G,k} ) \geq (1/2)^{t|S|} \geq (1/2)^{2t\ell}. $$
Writing $E_{S,G} := \bigcup_{k\in K} E_{S,G,k}$ and denoting by $A^c$ the complement of $A$, we find that
$$ 
\renewcommand*{\arraystretch}{1.2}
\begin{array}{rcl}
 \Pee( E_{S,G}^c ) 
 & \leq & 
 \left(1-(1/2)^{2\ell\cdot t}\right)^{|K|} \\
 & \leq & 
 \exp\left[ - |K| \cdot (1/2)^{2\ell\cdot t} \right] \\
 & = & 
 \exp\left[ - \exp\left[ \Omega( t \ln\ln n ) - O( t \ell ) \right] \right] \\
 & = & \exp\left[ - \exp\left[ \Omega( t \ln\ln n ) \right] \right]. 
  \end{array} 
  $$
Let $E$ denote the event that for every $S \subseteq C_0$ of the form $S = N(C_i) \cup N(C_j)$ with $|S| \leq 2\ell$ and 
for every target graph $G$ with $V(G)=S$ there is some $k>0$ such that $H(S,C_k) = G$.

We remark that there are at most $n^2$ choices of the set $S$ (as it must be a union $N(C_i) \cup N(C_j)$) and
at most $2^{{2\ell\choose 2}}$ choices of the target graph $G$.
Hence we have that 
$$ 
\renewcommand*{\arraystretch}{1.2}
\begin{array}{rcl}
   \Pee( E^c ) 
   & \leq & 
   n^2 \cdot 2^{{2\ell\choose 2}} \cdot \exp\left[ - \exp\left[ \Omega( t \ln\ln n ) \right] \right] \\
   & = & 
   \exp\left[ O( \ln n ) + O( \ell^2 ) - \exp\left[ \Omega( t \ln\ln n ) \right] \right] \\
   & = & 
   o(1),
   \end{array} $$
Since $\exp[ \Omega( t \ln\ln n) ] \gg \ln n$ ($\gg \ell^2$).
\end{proof}

We now have all the tools to prove Theorem~\ref{thm:undecidable}.

\begin{proofof}{Theorem~\ref{thm:undecidable}}
Let $\varphi \in \FO$ be an arbitrary sentence and let $\Phi(.,.)$ be as provided by Corollary~\ref{cor:transform}.
Consider the sentence
$$ \psi := \exists x,y : \Phi(x,y). 
$$
Up to error probability $o(1)$, we have that $\psi$ holds if and only if~there exist $i,j>0$ such that 
$H(N(C_i), C_j) \models \varphi$.

Thus, if $\spec(\varphi) = \emptyset$, that is if there is no finite graph that satisfies $\varphi$, then 
clearly $\Pee( G_n \models \psi ) = o(1)$.

On the other hand, if $\spec(\varphi) \neq \emptyset$, that is, if there is some finite graph $H$ such that 
$H \models \varphi$, then by Lemmas~\ref{lem:rightsize} and~\ref{lem:realise} for $n$ sufficiently large (namely $n$ such that $\ell \geq v(H)$)  
we will find indices $i,j>0$ such that $H(N(C_i), C_j) = H$.
This shows that, if $\spec(\varphi) \neq \emptyset$, then 
$\Pee( G_n \models \psi ) = 1 - o(1)$.

We have just shown that the constructed sentence $\psi$ has limiting probability zero if $\spec(\varphi)=\emptyset$ and limiting probability
one otherwise.
Thus any algorithm that can decide whether $\lim_{n\to\infty} \Pee( G_n \models \psi )$ equals zero or equals one
will allow us to decide whether or not $\varphi$ has a finite model. 
Therefore, there can be no such algorithm as this would contradict  Trakhtenbrot's theorem.
\end{proofof}

Before we can prove Theorem~\ref{thm:FO} we need one more ingredient.

\begin{corollary}\label{cor:bigger}
There exists an $\FO$-formula $\Bigger$ with two free variables such that, with high probability:
\begin{itemize}
 \item If $\Bigger(x,y)$ holds then there exist $i,j>0$ such that $x \in C_i, y \in C_j$ and $|N(C_i)|>|N(C_j)|$;
 \item If $x \in C_i, y \in C_j$ for some $i,j>0$ with $|N(C_j)|<|N(C_i)| \leq \ell$ then $\Bigger(x,y)$ holds.
\end{itemize}
\end{corollary}

\begin{proof}
The main idea behind the $\FO$-formula we'll give is that it expresses that there exists
a $k>0$ such that $H(C_j\Delta C_i, C_k)$ is a matching
between $C_j \setminus C_i$ and $C_i \setminus C_j$ that saturates all of $C_j\setminus C_i$, but there is at least
one unmatched vertex in $C_i \setminus C_j$.
The reader can check that the following formula will do the trick (assuming that $\InC0$ and $\CN$ express the correct thing):
$$ 
\renewcommand*{\arraystretch}{1.2}
\begin{array}{rcl} 
\Bigger(x,y) 
& := &  
 \neg \InC0(x) \wedge \neg \InC0(y) \wedge \neg(x=y) \wedge \neg (x\sim y) \\
 & &  \wedge (\exists z : (\forall y_1: (\CN(y_1,y) \wedge \neg \CN(y_1,x) ) \Rightarrow \\
 & & (\exists! x_1: \CN(x_1,x) \wedge \neg \CN(x_1, y) \wedge \Hedge(x_1,y_1,z) ) )  \\
 & &  \wedge (\forall x_1,y_1,y_2 : (\CN(x_1,x) \wedge \neg \CN(x_1, y) \wedge \CN(y_1,y) \\
 & & \wedge \neg \CN(y_1, x) \wedge \CN(y_2,y)  \wedge \neg \CN(y_2, x) \wedge \Hedge(x_1,y_1,z) \\
 & &  \wedge \Hedge(x_1,y_2,z)) \Rightarrow (y_1=y_2)) \\
 & &  \wedge (\exists x_1 : \CN(x_1,x) \wedge \neg \CN(x_1,y) \\
 & & \wedge (\forall y_1: \CN(y_1,y) \wedge \neg \CN(y_1,x) \Rightarrow \neg \Hedge(x_1,y_1,z)))).
   \end{array} $$
\end{proof}

We are now ready to prove the main result.

\begin{proofof}{Theorem~\ref{thm:FO}}
Let $\Phi_i$ denote the formula that Corollary~\ref{cor:transform} produces when applied to the 
sentence $\varphi_i$ from Lemma~\ref{lem:logloglemma}.
We define the following $\FO$-sentence:

$$ \varphi :=
 \exists x,y : \Phi_1(x,y) \wedge \neg( \exists x',y' : \Bigger(x',x) \wedge \Phi_0(x',y') ). 
$$
Up to error probability $o(1)$, the sentence $\varphi$ will hold if and only if~$H( N(C_i), C_j ) \models \varphi_1$ for some $i,j>0$,
and moreover if $H( N(C_{i'}), C_{j'} ) \not\models \varphi_0$ for some $i',j'>0$ then 
$\Bigger(x,x')$ does not hold for any $x\in C_i, x' \in C_{i'}$.
We briefly explain how this implies that $\varphi$ does not have a limiting probability.

First we consider an increasing subsequence $(n_k)_k$ of the natural numbers for which $\log^*n \mood 100 = 75$. 
Observe that 

\begin{equation}\label{eq:logsterel} \log^* n - 10 \leq \log^* \ell  \leq \log^* n. 
\end{equation}

With high probability there are lots of $C_i$ for which $|N(C_i)| = \ell$ by Lemma~\ref{lem:rightsize}, and by 
Lemma~\ref{lem:realise} for each of 
these there is a $j$ such that $H(N(C_i), C_j) \models \varphi_1$ (since $\ell \in \spec(\varphi_1)$ 
as $\log^*\ell \mood 100 \in \{65,\dots, 75\}$ by~\eqref{eq:logsterel} and the choice of $n$).
So there are (lots of) pairs of vertices $x,y$ such that $\Phi_1(x,y)$ holds and $x \in C_i$ for some $i>0$ with $|N(C_i)| = \ell$.
On the other hand, with high probability, for any $x'$ such that $\Bigger(x',x)$ it must hold that $x' \in C_{i'}$ for some $i'>0$ with 
$\ell = |N(C_{i})| < |N(C_{i'})| \leq n$. 
So in particular $\log^*(|N(C_{i'})|) \in \{65,\dots,75\}$.
Thus $|N(C_{i'})| \not\in \spec(\varphi_0)$, which shows that $H(N(C_{i'}), C_{j'}) \not\models\varphi_0$ for any $j'>0$.
In other words, if $\Bigger(x',x)$ holds then there cannot be any $y'$ such that $\Phi_0(x',y')$ holds.
This shows that
$$ \lim_{n\to\infty, \atop \log^*n \mood 100 = 75} \Pee( G_n \models \varphi ) = 1. 
$$

Next, let us consider an increasing subsequence $(n_k)_k$ of the natural numbers for which $\log^* n \mood 100 = 25$.
In this case $\log^*\ell \mood 100 \in \{ 15, \dots, 25 \}$.
In particular $\ell, \dots, n \not\in \spec(\varphi_1)$.
So, with high probability, if there is pair $x,y$ such that $\Phi_1(x,y)$ holds then 
we must have $x \in C_i$ for some $i>0$ with $|N(C_i)|$ strictly smaller than $\ell$.
But then we can again apply Lemma's~\ref{lem:rightsize} and~\ref{lem:realise} to find that, with high probability, 
there exist $x',y'$ with $x' \in C_{i'}, y \in C_{j'}$ for some $i',j'>0$ such that $|N(C_{i'})|=\ell$ and 
$H(N(C_{i'}),C_{j'}) \models \varphi_0$.
Since $\ell = |N(C_{i'})|>|N(C_i)|$, with high probability, $\Bigger(x',x)$ will hold by Corollary~\ref{cor:bigger}. 
This shows that
$$ \lim_{n\to\infty, \atop \log^*n \mood 100 = 25} \Pee( G_n \models \varphi ) = 0. $$
\end{proofof}

\section{Discussion and further work\label{sec:discuss}}

We remark that with very minor variations on our proofs, it can been seen that 
Theorems~\ref{thm:FO} and~\ref{thm:undecidable} also hold for random unipolar and random co-unipolar graphs.
Similarly, by a minor variation of the proof of Theorem~\ref{thm:undecidable}, it can be shown that it is undecidable to determine, given a
sentence $\varphi \in \FO$ that is guaranteed to have limiting probability $\in \{0,1/2\}$ (resp.~$\{1/2,1\}$), 
whether the limit is $1/2$.

Furthermore, by combining Corollary 10.37 from~\cite{purple-book} with a minor variation of our proof of Theorem~\ref{thm:undecidable}
it can be seen that there are formulas with a limiting probability but for which the convergence is extremely slow, in the following
precise sense. For every recursive function $f : \eN \to \eN$ and every $k$ there exists a $\varphi \in \FO$ of quantifier depth 
$\leq k$ (the definition of which can for instance be found in~\cite{purple-book}) such that $\lim_{n\to\infty} \Pee( G_n \models \varphi ) = 1$ yet $\max_{n \leq f(k)} \Pee( G_n \models \varphi ) = o(1)$. 

Recall that having a fixed graph $H$ as an induced subgraph will have limiting probability $0,1/2$ or $1$.
During the exploratory stages of the research that led to the present paper, the last two authors spent some effort trying to 
construct a $\FO$ sentence with a limiting probability $\not\in \{0,1/2,1\}$, without success.
We thus pose this as an open problem to which we would love to know the answer.

\medskip 

\begin{itemize} \item[] {\bf Question.} 
Does there exist a $\varphi \in \FO$ such that $\lim_{n\to\infty} \Pee( G_n \models \varphi )$ exists and takes
on a value other than $0,1/2$ or $1$?
\end{itemize}

\section*{Acknowlegdements}

We warmly thank Tomasz {\L}uczak for helpful discussions which have greatly improved the paper. 
Amongst other things, Tomasz {\L}uczak suggested Theorem~\ref{thm:undecidable} and its proof to us.

This research was started during the BGSMath research program on Random Discrete Structures and Beyond 
that was held in Barcelona from May to June 2017.

\end{document}